\documentclass[11pt,reqno]{amsart}

\usepackage{hyperref}
\usepackage[alphabetic,lite]{amsrefs}
\usepackage{dlfltxbcodetips}
\usepackage{verbatim}
\usepackage{amscd}   
\usepackage[all]{xy} 
\usepackage{youngtab} 
\usepackage{young} 
\usepackage{tikz}
\usepackage{ mathrsfs }
\usepackage{cases}

\newcommand{\defi}[1]{{\upshape\sffamily #1}}

\renewcommand{\a}{\alpha}
\renewcommand{\b}{\beta}
\renewcommand{\c}{\gamma}
\renewcommand{\d}{\delta}
\newcommand{\D}{\mathcal{D}}
\newcommand{\bw}{\bigwedge}
\renewcommand{\det}{\textrm{det}}
\newcommand{\G}{\Gamma}

\renewcommand{\ll}{\lambda}

\newcommand{\onto}{\twoheadrightarrow}
\newcommand{\oo}{\otimes}

\renewcommand{\P}{\mathcal{P}}
\newcommand{\s}{\sigma}
\renewcommand{\S}{\mathcal{S}}
\renewcommand{\t}{\underline{t}}
\newcommand{\V}{\mathcal{V}}
\newcommand{\x}{\underline{x}}
\newcommand{\y}{\underline{y}}
\newcommand{\z}{\underline{z}}

\newcommand{\Ext}{\operatorname{Ext}}
\newcommand{\GL}{\operatorname{GL}}

\newcommand{\Sym}{\operatorname{Sym}}

\newcommand{\bb}[1]{\mathbb{#1}}

\renewcommand{\rm}[1]{\textrm{#1}}
\newcommand{\mc}[1]{\mathcal{#1}}
\newcommand{\mf}[1]{\mathfrak{#1}}

\newcommand{\tl}[1]{\tilde{#1}}
\newcommand{\ul}[1]{\underline{#1}}
\newcommand{\scpr}[2]{\left\langle #1,#2 \right\rangle}

\newtheorem{theorem}{Theorem}[section]
\newtheorem*{theorem*}{Theorem}
\newtheorem{lemma}[theorem]{Lemma}

\newtheorem{proposition}[theorem]{Proposition}

\newtheorem*{corollary*}{Corollary}

\newtheorem*{charext*}{Theorem on the Characters of Ext Modules}
\newtheorem*{ext*}{Theorem on the Growth of Ext Modules}
\newtheorem*{main*}{Main Theorem}
\newtheorem*{regularity*}{Theorem on Regularity of Equivariant Ideals}
\newtheorem*{nonvanishing*}{Theorem on Non-vanishing of Local Cohomology with Determinantal Support}

\theoremstyle{definition}

\newtheorem*{definition*}{Definition}
\newtheorem{example}[theorem]{Example}

\theoremstyle{remark}
\newtheorem{remark}[theorem]{Remark}
\newtheorem*{remark*}{Remark}

\numberwithin{equation}{section}

\begin{document}

\title{Local cohomology with support in ideals of symmetric minors and Pfaffians}

\author{Claudiu Raicu}
\address{Department of Mathematics, University of Notre Dame, 255 Hurley, Notre Dame, IN 46556\newline
\indent Institute of Mathematics ``Simion Stoilow'' of the Romanian Academy}
\email{craicu@nd.edu}

\author{Jerzy Weyman}
\address{Department of Mathematics, University of Connecticut, Storrs, CT 06269}
\email{jerzy.weyman@uconn.edu}


\date{\today}


\begin{abstract} 
 We compute the local cohomology modules $\mc{H}_Y^{\bullet}(X,\mc{O}_X)$ in the case when $X$ is the complex vector space of $n\times n$ symmetric, respectively skew-symmetric matrices, and $Y$ is the closure of the $\GL$-orbit consisting of matrices of any fixed rank, for the natural action of the general linear group $\GL$ on $X$. We describe the $\D$-module composition factors of the local cohomology modules, and compute their multiplicities explicitly in terms of generalized binomial coefficients. One consequence of our work is a formula for the cohomological dimension of ideals of even minors of a generic symmetric matrix: in the case of odd minors, this was obtained by Barile in the 90s. Another consequence of our work is that we obtain a description of the decomposition into irreducible $\GL$-representations of the local cohomology modules (the analogous problem in the case when $X$ is the vector space of $m\times n$ matrices was treated in earlier work of the authors).
\end{abstract}

\maketitle

\section{Introduction}\label{sec:intro}

In this paper we consider $X$ to be a vector space of matrices (general, symmetric, or skew-symmetric). We let $Y_p$ denote the subvariety of general or symmetric matrices of rank at most $p$, or of skew-symmetric matrices of rank at most $2p$. We write $\D_X$ for the Weyl algebra of differential operators on $X$, and $D_p=\mc{L}(Y_p,X)$ for the intersection homology $\D_X$-module corresponding to $Y_p$. Our goal is to compute the composition factors of the local cohomology modules $\mc{H}^{\bullet}_{Y_p}(X,\mc{O}_X)$: we express, in the Grothendieck group of holonomic $\D_X$-modules, the local cohomology modules in terms of the modules $D_p$. As a corollary, we obtain descriptions of the local cohomology modules as representations of the general linear group (we will refer to these as the \defi{characters} of local cohomology modules). We have already performed the character calculation in the case of general matrices in \cite{raicu-weyman}, so here we only formulate the $\D$-module version of the main result of \cite{raicu-weyman}.

We write $\G(X)$ for the Grothendieck group of holonomic $\D_X$-modules, and let $q$ denote a formal variable. We write $[\mc{M}]$ for the class of a $\D_X$-module $\mc{M}$ in $\G(X)$ and define $H_p(q)\in\G(X)[q]$ via
\begin{equation}\label{eq:defHpq}
 H_p(q)=\sum_{j\geq 0}[\mc{H}_{Y_p}^j(X,\mc{O}_X)]\cdot q^j.
\end{equation}
For nonnegative integers $a\geq b$ we define the \defi{Gauss polynomials} (or \defi{$q$-binomial coefficients}) ${a\choose b}_q$ by
\begin{equation}\label{eq:defqbinomial}
{a\choose b}_q=\frac{(1-q^a)(1-q^{a-1})\cdots(1-q^{a-b+1})}{(1-q^b)(1-q^{b-1})\cdots(1-q)}. 
\end{equation}

\begin{main*}
 With the notation above, we have the following expressions for $H_p(q)$:
 \begin{itemize}
  \item If $X=\bb{C}^{m\times n}$ is the space of $m\times n$ matrices, $m\geq n$, then for $0\leq p<n$
  \begin{equation}\label{eq:Hpqgeneral}
   H_p(q)=\sum_{s=0}^p [D_s]\cdot q^{(n-p)^2+(n-s)\cdot(m-n)}\cdot{n-s-1\choose p-s}_{q^2}.
  \end{equation}
  \item If $X=\bw^2\bb{C}^n$ is the space of $n\times n$ skew-symmetric matrices and $m=\lfloor n/2\rfloor$, then for $0\leq p<m$
  \begin{equation}\label{eq:Hpqskew}
   H_p(q)=\begin{cases}
    \displaystyle\sum_{s=0}^p [D_s]\cdot q^{2(m-p)^2+(m-p)+2(p-s)}\cdot{m-1-s\choose p-s}_{q^4} & \rm{ if }n=2m+1\rm{ is odd;} \\ 
    \displaystyle\sum_{s=0}^p [D_s]\cdot q^{2(m-p)^2-(m-p)}\cdot{m-1-s\choose p-s}_{q^4} & \rm{ if }n=2m\rm{ is even.} \\ 
   \end{cases}
  \end{equation}
  \item If $X=\Sym^2\bb{C}^n$ is the space of $n\times n$ symmetric matrices, then for $0\leq p<n$
  \begin{equation}\label{eq:Hpqsymm}
   H_p(q)=\sum_{\substack{s=0 \\ s\equiv p\ (\rm{mod }2)}}^p [D_s]\cdot q^{1+{n-s+1\choose 2}-{p-s+2\choose 2}}\cdot{\lfloor\frac{n-s-1}{2}\rfloor\choose\frac{p-s}{2}}_{q^{-4}}.
  \end{equation}
 \end{itemize}
\end{main*}

An immediate corollary of the Main Theorem is a formula for the \defi{local cohomological dimension of~$X$ along~$Y_p$}, defined via
\[\rm{lcd}(X;Y_p)=\max\{j:\mc{H}_{Y_p}^j(X,\mc{O}_X)\neq 0\}.\]
Using (\ref{eq:defHpq}), $\rm{lcd}(X;Y_p)$ is precisely the highest exponent of $q$ appearing in the polynomial $H_p(q)$. It follows that for $X=\bb{C}^{m\times n}$, $\rm{lcd}(X;Y_p)=mn-(p+1)^2+1$ \cite{bru-sch}, and for $X=\bw^2\bb{C}^n$, $\rm{lcd}(X;Y_p)={n\choose 2}-{2p+2\choose 2}+1$ \cite[Thm.~6.1]{barile}. In the case when $X=\Sym^2\bb{C}^n$ we get
\begin{equation}\label{eq:cdimsymm}
  \rm{lcd}(X,Y_p)=\begin{cases}
1+{n+1\choose 2}-{p+2\choose 2} & \rm{if }p\rm{ is even;} \\
1+{n\choose 2}-{p+1\choose 2} & \rm{if }p\rm{ is odd.} \\
\end{cases}
\end{equation}
The formula (\ref{eq:cdimsymm}) was known in the case when $p$ is even \cite[Thm.~6.3]{barile}, but it is new in the case when $p<n-1$ is odd, as far as we are aware! We point out that the previously known cohomological dimension calculations were based on the fact that the top non-vanishing local cohomology module is supported at $Y_0=\{0\}$, in which case it can be determined via a computation of singular cohomology \cite[Thm.~3.1]{lyubeznik-singh-walther}. In contrast, in the case when $X=\Sym^2\bb{C}^n$ and $p<n-1$ is odd, the top non-vanishing local cohomology module $\mc{H}_{Y_p}^{\bullet}(X,\mc{O}_X)$ is the simple $\D$-module $D_1$, which is supported on $Y_1$. Using \cite[Thm.~3.1]{barile} we conclude that the case when $X=\Sym^2\bb{C}^n$ and $p<n-1$ is odd is also the only one where $\rm{lcd}(X;Y_p)$ differs from the \defi{arithmetic rank of $Y_p\subset X$} (i.e. from the minimal number of equations required to cut out $Y_p$ as an algebraic subset of $X$).

The formula (\ref{eq:Hpqgeneral}) is a direct consequence of \cite{raicu-weyman}, as explained in \cite[Remark~1.4]{raicu-dmods}: it is included in the Main Theorem for the sake of completeness, and will not be addressed further. The formulas (\ref{eq:Hpqskew}--\ref{eq:Hpqsymm}) give rise to expressions for the characters of the local cohomology modules with support in ideals of symmetric minors/Pfaffians, as we explain next. We let $\bb{Z}^n_{dom}$ denote the set of \defi{dominant weights} $\ll=(\ll_1\geq\cdots\geq\ll_n)\in\bb{Z}^n$, consider the sets
\[
\begin{aligned}
\mc{B}(s,2m) &= \{\ll\in\bb{Z}^{2m}_{dom}:\ll_{2s}\geq(2s-1),\ll_{2s+1}\leq 2s,\ll_{2i-1}=\ll_{2i}\rm{ for all }i\}, \\
\mc{B}(s,2m+1) &= \{\ll\in\bb{Z}^{2m+1}_{dom}:\ll_{2s+1}=2s,\ll_{2i-1}=\ll_{2i}\rm{ for }i\leq s,\ll_{2i}=\ll_{2i+1}\rm{ for }i>s\},
\end{aligned}
\]
\[
\begin{aligned}
\mc{C}^1(s,n)&=\{\ll\in\bb{Z}^{n}_{dom}:\ll_i\overset{(\rm{mod }2)}{\equiv} s+1\rm{ for }i=1,\cdots,n,\ll_{s}\geq s+1\geq\ll_{s+2}\},\\
\mc{C}^2(s,n)&=\left\{\ll\in\bb{Z}^{n}_{dom}:\ll_i\overset{(\rm{mod }2)}{\equiv}
\begin{cases}
 s+1 & \rm{for }i=1,\cdots,s\\
 s & \rm{for }i=s+1,\cdots,n
\end{cases},
\ll_{s}\geq s+1,\ll_{s+1}\leq s
\right\},
\end{aligned}
\]
and use them to define the $\GL_n(\bb{C})$-representations (where $S_{\ll}$ denotes the \defi{Schur functor} associated to $\ll$)
\begin{equation}\label{eq:defBsCsj}
 \mf{B}_s=\bigoplus_{\ll\in\mc{B}(s,n)} S_{\ll}\bb{C}^n,\ s=0,\cdots,\lfloor n/2\rfloor,\quad\rm{and}\quad\mf{C}^j_s=\bigoplus_{\ll\in\mc{C}^j(s,n)} S_{\ll}\bb{C}^n,\ s=0,\cdots,n. 
\end{equation}

The following is a direct consequence of Theorems~4.1 and~5.1 in \cite{raicu-dmods}:
\begin{theorem}\label{thm:charsofICmods}
 If $X=\bw^2\bb{C}^n$ is the vector space of $n\times n$ skew-symmetric matrices then we have the isomorphisms of $\GL_n(\bb{C})$-representations
 \[D_s\simeq\mf{B}_{m-s}\rm{ for }s=0,\cdots,m=\lfloor n/2\rfloor.\]
 If $X=\Sym^2\bb{C}^n$ is the vector space of $n\times n$ symmetric matrices then we have similarly
 \[D_s\simeq\begin{cases}
 \mf{C}_{n-s}^1 & \rm{if }n-s\rm{ is odd}; \\            
 \mf{C}_{n-s}^2 & \rm{if }n-s\rm{ is even}. \\            
\end{cases}
 \]
\end{theorem}
Replacing the $\D$-modules $D_s$ in (\ref{eq:Hpqskew}) and (\ref{eq:Hpqsymm}) by the corresponding $\GL_n(\bb{C})$-representations as in Theorem~\ref{thm:charsofICmods}, we obtain a description of the characters of local cohomology modules with support in ideals of symmetric minors/Pfaffians. These formulas give the analogue in the symmetric/skew-symmetric case of the main result of \cite{raicu-weyman}. Note that the special case of (\ref{eq:Hpqskew}) when $n=2m+1$ is odd and $p=m-1$ (that of sub-maximal Pfaffians) recovers \cite[Thm.~5.5]{raicu-weyman-witt}.

\subsection*{Strategy for proving the Main Theorem}

For $X$ the space of symmetric/skew-symmetric matrices, we let $(F_i)_{i\in I}$ denote the collection of simple $\GL$-equivariant holonomic $\D_X$-modules. By \cite[Prop.~3.1.2]{VdB:loccoh}, the local cohomology modules with determinantal/Pfaffian support have a composition series whose composition factors are the modules $F_i$, each appearing with some (finite) multiplicity. It follows that for fixed $p$ there exist polynomials $P_i\in\bb{Z}[q]$, and a relation inside $\G(X)[q]$
\begin{equation}\label{eq:DmodHpqtoFi}
H_p(q)=\sum_{i\in I} [F_i]\cdot P_i(q).
\end{equation}
Using the forgetful map from $\G(X)$ to the Grothendieck group $\G(\GL)$ of admissible $\GL$-representations, we can interpret (\ref{eq:DmodHpqtoFi}) as an equality in $\G(\GL)[q]$. For each $i\in I$, it is possible to choose a \defi{witness weight} $\ll(F_i)\in\bb{Z}^n_{dom}$, with the property that $S_{\ll(F_i)}\bb{C}^n$ appears as a subrepresentation inside $F_j$ if and only if $i=j$. To determine the polynomials $P_i(q)$ it is then sufficient to determine the multiplicity of $S_{\ll(F_i)}\bb{C}^n$ inside the local cohomology modules $\mc{H}^{\bullet}_{Y_p}(X,\mc{O}_X)$.

At this point we can apply the strategy from \cites{raicu-weyman,raicu-VeroDmods}: we write an appropriate spectral sequence that computes local cohomology and in addition it is $\GL$-equivariant. The terms of the spectral sequence are computed using a Grothendieck duality argument and Bott's Theorem. When restricting the spectral sequence to the $\ll(F_i)$-isotypic component we notice that there can be no non-zero $\GL$-equivariant maps, i.e. the restriction of the spectral sequence is degenerate. This allows us to compute the multiplicities of the witness weights inside the local cohomology modules $\mc{H}^{\bullet}_{Y_p}(X,\mc{O}_X)$, and thus to determine the polynomials $P_i(q)$. We note that for skew-symmetric matrices the spectral sequence is degenerate even before the restriction to the witness weights, so the approach of \cite{raicu-weyman} goes through with little modifications. However, this is not the case for symmetric matrices: we were not able to write down a degenerate spectral sequence for computing local cohomology, so the $\D$-module insight is essential in this case. Even for skew-symmetric matrices, the $\D$-module approach offers a simplification of the arguments of \cite{raicu-weyman}, in that it is not necessary to compute all the terms in the spectral sequence, but only their isotypic components corresponding to the witness weights. We encourage the interested reader to apply the $\D$-module/witness weight arguments below in the case when $X$ is the space of general matrices (\ref{eq:Hpqgeneral}), as an alternative to the calculations in~\cite{raicu-weyman}.

\subsection*{Organization} In Section~\ref{subsec:repthybott} we fix some notation regarding the representation theory of $\GL_n(\bb{C})$, and describe a useful consequence of Bott's Theorem for Grassmannians. In Section~\ref{subsec:charDmods} we recall the character calculations for equivariant $\D$-modules on (skew-)symmetric matrices from \cite{raicu-dmods}. In Section~\ref{subsec:spseq} we set up a spectral sequence for computing local cohomology with symmetric determinantal/Pfaffian support, whose terms are $\Ext$ modules computed in Section~\ref{subsec:ext} via the Duality Theorem \cite[Thm.~3.1]{raicu-weyman-witt}. In Sections~\ref{sec:skew} and~\ref{sec:symm} we perform the local cohomology calculation following the strategy outlined above.

\section{Preliminaries}\label{sec:prelim}

\subsection{Representation theory and Bott's Theorem}\label{subsec:repthybott}

Given a complex vector space $W$ of dimension $n$, we write $\GL(W)$ for its group of invertible linear transformations. The irreducible representations of $\GL(W)$ are classified by $\bb{Z}^n_{dom}$, the set of \defi{dominant weights} $\ll=(\ll_1\geq\cdots\geq\ll_n)\in\bb{Z}^n$. We write $S_{\ll}W$ for the irreducible corresponding to $\ll$. We write $\det(W)=\bw^n W$ for the irreducible $S_{\ll}W$ corresponding to $\ll=(1,\cdots,1)=(1^n)$. The \defi{dual weight} $\ll^*$ of a weight $\ll$ is defined by $\ll^*_i=-\ll_{n+1-i}$. If we let $V=W^*$ be the dual vector space to $W$ then $S_{\ll}V=S_{\ll^*}W$. The \defi{size} of a weight $\ll$ is $|\ll|=\ll_1+\cdots+\ll_n$. A \defi{partition} is a dominant weight whose entries are non-negative integers (we often use underlined roman letters for partitions, to distinguish them from the arbitrary dominant weights, for which we use greek letters). We let $\P(k)$ denote the set of partitions $\z=(z_1\geq\cdots\geq z_k\geq 0)$ with at most $k$ parts, and whenever necessary we identify $\P(k)$ with a subset of $\P(k+1)$. We write $\P(k,n-k)$ for the subset of partitions $\z\in\P(k)$ satisfying $z_1\leq n-k$. The \defi{conjugate partition} to $\z$ is denoted by $\z'$, so that $\P(n-k,k)=\{\z':\z\in\P(k,n-k)\}$. The Gauss polynomials defined in~(\ref{eq:defqbinomial}) can be thought of as generating functions for the number of partitions of any given size in $\P(a-b,b)$:
\begin{equation}\label{eq:genfunqbin}
{a\choose b}_q=\sum_{\z\in\P(a-b,b)}q^{|\z|}. 
\end{equation}
It will be useful to note that taking the complement of a Young diagram inside the $(a-b)\times b$ rectangle establishes a bijection from $\P(a-b,b)$ to itself, so we can rewrite (\ref{eq:genfunqbin}) as
\begin{equation}\label{eq:genfunqbincomplement}
{a\choose b}_q=\sum_{\z\in\P(a-b,b)}q^{(a-b)\cdot b-|\z|}.
\end{equation}

For a partition $\z\in\P(k)$, we write $\z^{(2)}=(z_1,z_1,z_2,z_2,\cdots)\in\P(2k)$ for the partition obtained by repeating the parts of $\z$. We write $2\z=(2z_1,2z_2,\cdots)\in\P(k)$ for the one obtained by doubling the parts of $\z$. By \cite[Prop.~2.3.8]{weyman}, we have the following formulas which will be used throughout the next sections:
\begin{equation}\label{eq:cauchy}
\Sym(\Sym^2 V)=\bigoplus_{\z\in\P(n)} S_{2\z}V,\quad\Sym\left(\bw^2 V\right)=\bigoplus_{\z\in\P(\lfloor n/2\rfloor)} S_{\z^{(2)}}V. 
\end{equation}
If we write $S$ for $\Sym(\Sym^2 V)$ (resp. $\Sym\left(\bw^2 V\right)$) then for any partition $\z$ in $\P(n)$ (resp. $\P(\lfloor n/2\rfloor)$) we define
\begin{equation}\label{eq:defIz}
 I_{\z}=\rm{the ideal generated by }S_{2\z}V\ (\rm{resp. by }S_{\z^{(2)}}V).
\end{equation}
We define a partial order on the partitions by $\y\leq\z$ if and only if $y_i\leq z_i$ for all $i$. It follows from \cites{abeasis,abeasis-delfra} that $I_{\y}\subseteq I_{\z}$ if and only if $\y\geq\z$, or equivalently
\begin{equation}\label{eq:descIzs}
 I_{\z}=\bigoplus_{\y\geq\z} S_{2\y}V\rm{ when }S=\Sym(\Sym^2 V),\rm{ and }I_{\z}=\bigoplus_{\y\geq\z} S_{\y^{(2)}}V\rm{ when }S=\Sym\left(\bw^2 V\right).
\end{equation}
Analogous statements hold for general matrices (see \cite{DCEP}), and this was employed by the authors in the local cohomology calculations from \cite{raicu-weyman}.

%

If $U=\bigoplus_{\ll}(S_{\ll}W)^{\oplus a_{\ll}}$ is a representation of $\GL(W)$, we write
\begin{equation}\label{eq:defscpr}
\scpr{S_{\ll}W}{U}=a_{\ll} 
\end{equation}
for the multiplicity of the irreducible $S_{\ll}W$ inside $U$. We call the subrepresentation $(S_{\ll}W)^{\oplus a_{\ll}}$ the \defi{$\ll$-isotypic component} of $U$. For a cohomologically $\bb{Z}$-graded $\GL(W)$-representation $H^{\bullet}=\bigoplus_{j\in\bb{Z}}H^j$, and a dominant weight $\ll\in\bb{Z}^n_{dom}$, we define a generating function for the multiplicities of the $\ll$-isotypic components of $H^{\bullet}$:
\begin{equation}\label{eq:genfcngeneric}
\scpr{H^{\bullet}}{S_{\ll}W}=\sum_{j\in\bb{Z}}\scpr{H^j}{S_{\ll}W}\cdot q^j. 
\end{equation}

The following consequence of Bott's Theorem \cite[Cor.~4.1.9]{weyman} will be useful later:

\begin{theorem}\label{thm:bott}
 Consider the Grassmannian $\bb{G}=\bb{G}(k,V)$ of $k$-dimensional quotients of $V=W^*$, and let $\mc{Q}$ denote the tautological rank $k$ quotient bundle, and $\mc{R}$ the tautological rank $(n-k)$ sub-bundle.
 
 (a) If $\a\in\bb{Z}^k_{dom}$ and $\b\in\P(n-k)$ is a partition then
 \[\scpr{H^{\bullet}(\bb{G},\mc{S}_{\b}\mc{R}\oo S_{\a}\mc{Q})}{\bb{C}}=
\begin{cases}
 q^{|\b|} & \rm{if }\b\in P(n-k,k)\rm{ and }\a=(\b')^*; \\
 0 & \rm{otherwise}.
\end{cases}
 \]
 
 (b) If $\a\in\bb{Z}^k_{dom}$, $n-k\leq s\leq n$, and $\b\in\P(n-k)$ is a partition with $\b_i\geq n-s$ for all $i$, then
 \[\scpr{H^{\bullet}(\bb{G},\mc{S}_{\b}\mc{R}\oo S_{\a}\mc{Q})}{\bw^{n-s} W}=
\begin{cases}
 q^{|\b|} & \rm{if }\b\in P(n-k,k)\rm{ and }\a=(\b')^*+(0^{s-n+k},-1^{n-s}); \\
 0 & \rm{otherwise}.
\end{cases}
 \]
\end{theorem}

\begin{proof}
 Observe first that it suffices to prove part (b): (a) is the special case when $s=n$. We write $\c=(\c_1,\cdots,\c_n)=(\a_1,\cdots,\a_k,\b_1,\cdots,\b_{n-k})$ for the concatenation of $\a$ and $\b$, let $\delta=(n-1,n-2,\cdots,0)$, and consider $\c+\delta=(\c_1+n-1,\c_2+n-2,\cdots,\c_n)$. We write $\tl{\c}=\rm{sort}(\c+\delta)$ for the sequence obtained by arranging the entries of $\c+\delta$ in non-increasing order. If we write $l$ for the number of pairs $(x,y)$ with $1\leq x<y\leq n$ and $\c_x-x<\c_y-y$, then Bott's Theorem yields
\begin{equation}\label{eq:bott}
 H^{\bullet}(\bb{G},\mc{S}_{\b}\mc{R}\oo S_{\a}\mc{Q})=\begin{cases}
 S_{\tl{\c}-\delta}V & \rm{ if }\c+\d\rm{ has distinct entries and }\bullet=l; \\
 0 & \rm{otherwise}.
\end{cases}
\end{equation}
Since $\bw^{n-s}W=S_{(0^s,-1^{n-s})}V$, in order to prove (b) we need to show the equivalence
\begin{equation}\label{eq:equiv1}
\tl{\c}-\d=(0^s,-1^{n-s})\Longleftrightarrow \b\in P(n-k,k)\rm{ and }\a=(\b')^*+(0^{s-n+k},-1^{n-s}), 
\end{equation}
and moreover if the equivalent conditions above hold, then $l=|\b|$.

Consider a permutation $\s$ of $[n]$ such that
\begin{equation}\label{eq:defsigma}
\tl{\c}_{\s(i)}=\c_i+\d_i,\rm{ for }i=1,\cdots,n, 
\end{equation}
and note that $\s$ is unique if $\c+\d$ has distinct entries. Since $\c_1\geq\cdots\geq\c_k$ and $\c_{k+1}\geq\cdots\geq\c_n$, we get that
\[\s(1)<\cdots<\s(k),\rm{ and }\s(k+1)<\cdots<\s(n).\]
Such $\s$ corresponds to a unique partition $\ul{t}\in\P(n-k,k)$ such that (recalling $\ul{t}'\in\P(k,n-k)$)
\begin{equation}\label{eq:ttosigma}
\s(i)=\begin{cases}
t'_{k+1-i}+i & \rm{for }i=1,\cdots,k; \\
i-t_{i-k} & \rm{for }i=k+1,\cdots,n.
\end{cases}
\end{equation}
Note that the condition $\c_x-x<\c_y-y$ is equivalent to $\tl{\c}_{\s(x)}<\tl{\c}_{\s(y)}$. If $\c+\d$ has non-repeated entries (i.e. $\tl{\c}$ is strictly decreasing), we get that
\begin{equation}\label{eq:countxy}
l=\#\{(x,y):x<y,\s(x)>\s(y)\}=|\ul{t}|. 
\end{equation}
We are now ready to prove the equivalence (\ref{eq:equiv1}), under the assumption in the statement of the theorem that $s\geq n-k$ and all $\b_i\geq n-s$.

$\Longleftarrow$: The condition $\b_i\geq n-s$ for all $i$ implies $\b'_1=\cdots=\b'_{n-s}=n-k$, so
\[\a=(-\b'_k,\cdots,-\b'_{n-s+1},-(n-k)-1,\cdots,-(n-k)-1).\]
If we let $k'=k-(n-s)\geq 0$ then we get
\[\c+\d=(n-1-\b'_k,\cdots,n-k'-\b'_{k+1-k'},n-s-2,n-s-3,\cdots,-1,n-k-1+\b_1,\cdots,\b_{n-k}).\]
If we take $\t=\b$ and define $\s$ via (\ref{eq:ttosigma}) then we get $\s(i)=n-k+i$ for $i=k'+1,\cdots,k$, and
\[\c+\d=(n-\s(1),\cdots,n-\s(k'),n-\s(k'+1)-1,\cdots,n-\s(k)-1,n-\s(k+1),\cdots,n-\s(n)).\]
Since $\s$ is a permutation of $[n]$, the numbers $n-\s(i)$, $i=1,\cdots,n$ form a permutation of $\{0,\cdots,n-1\}$. Moreover, the (ordered) sets $\{n-\s(k'+1),\cdots,n-\s(k)\}$ and $\{n-s-1,\cdots,0\}$ coincide. It follows that the entries of $\c+\d$ are all the integers from $-1$ to $n-1$, except for $n-s-1$, and in particular they are distinct. If we define $\tl{\c}_{\s(i)}=\c_i+\d_i=n-\s(i)$ or $n-\s(i)-1$, then it is clear that $\tl{\c}$ is non-increasing, and 
\begin{equation}\label{eq:tlc}
\tl{\c}=\rm{sort}(\c+\d)=(n-1,\cdots,n-s,n-s-2,\cdots,-1), 
\end{equation}
which yields $\tl{\c}-\d=(0^s,-1^{n-s})$, proving the implication. Furthermore, since $\s$ corresponds via (\ref{eq:ttosigma}) to $\t=\b$, it follows from (\ref{eq:countxy}) that $l=|\b|$, as desired.

$\Longrightarrow$: The condition $\tl{\c}-\d=(0^s,-1^{n-s})$ is equivalent to (\ref{eq:tlc}), which implies that $\c+\d$ has distinct entries, so there exists a unique permutation $\s$ for which (\ref{eq:defsigma}) is satisfied. Let $\t\in\P(n-k,k)$ be the partition corresponding to $\s$ as in (\ref{eq:ttosigma}). We begin by proving that $\ul{t}=\b$, which implies $\b\in\P(n-k,k)$.

Assume that $\s(i)=i-t_{i-k}>s$ for some $i=k+1,\cdots,n$. It follows from (\ref{eq:defsigma}) and (\ref{eq:tlc}) that
\begin{equation}\label{eq:btot}
\b_{i-k}+(n-i)=\c_i+\d_i=\tl{\c}_{\s(i)}=n-1-\s(i)=n-1-i+t_{i-k}, 
\end{equation}
so $\b_{i-k}=t_{i-k}-1$. We then get a contradiction since $n-s\leq\b_{i-k}=t_{i-k}-1<i-s-1$, which is equivalent to $n<i-1$. It follows that $\s(i)=i-t_{i-k}\leq s$ for all $i=k+1,\cdots,n$. A calculation similar to (\ref{eq:btot}), using the fact that $\tl{\c}_{\s(i)}=n-\s(i)$ when $\s(i)\leq s$, yields $\b_{i-k}=t_{i-k}$ for $i=k+1,\cdots,n$, i.e. $\ul{t}=\b$, as desired.

If we let $k'=k-(n-s)$, we obtain just as in the previous implication that $\s(i)=n-k+i>s$ for $i=k'+1,\cdots,k$. Note that for such $i$ we have $k+1-i\leq n-s$, so $\b'_{k+1-i}=n-k$ (recall that the condition $\b_i\geq n-s$ is equivalent to $\b'_1=\cdots=\b'_{n-s}=n-k$). It follows that for $i=k'+1,\cdots,k$,
\[\a_i=\c_i=\tl{\c}_{\s(i)}-\d_i=(n-1-\s(i))-(n-i)=i-1-\s(i)=i-1-(n-k+i)=-\b'_{k+1-i}-1.\]
This equality shows that the last $(n-s)$ entries of $\a$ and $(\b')^*$ agree, so it remains to check that $\a_i=-\b'_{k+1-i}$ for $i\leq k'$. To do so we note that $\s(i)\leq s$ for $i\leq k'$: this is because $\s$ is a permutation, and the numbers $s+1,\cdots,n$ are all of the form $\s(i)$ for some $i=k'+1,\cdots,k$. It follows that $\tl{\c}_{\s(i)}=n-\s(i)$ and thus
\[\a_i=\c_i=\tl{\c}_{\s(i)}-\d_i=(n-\s(i))-(n-i)=i-\s(i)=i-(t'_{k+1-i}+i)=-t'_{k+1-i}=-\b'_{k+1-i}.\qedhere\]
\end{proof}

\subsection{Characters of equivariant $\D$-modules and witness weights {\cite{raicu-dmods}}}\label{subsec:charDmods}

We write $W=\bb{C}^n$ and identify $\bw^2 W$ with the vector space of $n\times n$ skew-symmetric matrices, and $\Sym^2 W$ with the space of $n\times n$ symmetric matrices. Recalling the notation (\ref{eq:defBsCsj}), we get the following \cite[Theorems~4.1 and~5.1]{raicu-dmods}:

\begin{theorem}\label{thm:charsDmods}
 (a) There exist $m=\lfloor n/2\rfloor$ simple $\GL(W)$-equivariant holonomic $\D$-modules on $\bw^2 W$, whose characters are $\mf{B}_s$, $s=0,\cdots,m$. The support of the $\D$-module with character $\mf{B}_s$ consists of matrices of rank at most $m-s$.
 
 (b) There exist $2n+1$ simple $\GL(W)$-equivariant holonomic $\D$-modules on $\Sym^2 W$, whose characters are $\mf{C}_s^j$, $s=0,\cdots,n$, $j=1,2$ (note that $\mf{C}_n^1=\mf{C}_n^2$ correspond to the same $\D$-module). The support of the $\D$-module with character $\mf{C}_s^j$ consists of matrices of rank at most $n-s$.
\end{theorem}

We define the weights $\ll(\mf{B}_s)$ for $s=0,\cdots,m$, and $\ll(\mf{C}_s^j)$ for $s=0,\cdots,n$, $j=1,2$, via
\begin{equation}\label{eq:defllBs}
 \ll(\mf{B}_s)_i=2s\rm{ for }i=1,\cdots,n,
\end{equation}
\begin{equation}\label{eq:defllCsj}
 \ll(\mf{C}_s^1)_i=s+1\rm{ for }i=1,\cdots,n,\rm{ and }\ll(\mf{C}_s^2)_i=\begin{cases}
 s+1 & \rm{for }i=1,\cdots,s\\
 s & \rm{for }i=s+1,\cdots,n.
\end{cases}
\end{equation}
The weights $\ll(\mf{B}_s)$ (resp. $\ll(\mf{C}_s^j)$) are \defi{witness weights} for the simple equivariant $\D$-modules on the spaces of skew-symmetric (resp. symmetric) matrices in the following sense: the $\ll(\mf{B}_s)$-isotypic component of $\mf{B}_{s'}$ is non-zero if and only if $s=s'$; likewise, the $\ll(\mf{C}_s^j)$-isotypic component of $\mf{C}_{s'}^{j'}$ is non-zero if and only if $s=s'$ and $j=j'$ (or $s=s'=n$).

\subsection{Ext modules for the subquotients $J_{\x,p}$}\label{subsec:ext}

In this section $V$ denotes a complex vector space of dimension $n$. The next two results concern the calculation of certain $\Ext$ modules which appear in the spectral sequence for computing local cohomology.

\begin{lemma}\label{lem:Jxpsymm}
 Let $S=\Sym(\Sym^2 V)$. For $1\leq p\leq n$ consider a partition $\x=(x_1=x_2=\cdots=x_p\geq x_{p+1}\geq\cdots\geq x_n\geq 0)$. There exists a unique $\GL$-equivariant $S$-module $J_{\x,p}^{symm}$ with the properties
\begin{itemize}
 \item[(a)] As a $\GL$-representation, $J_{\x,p}^{symm}$ has a decomposition
 \[J_{\x,p}^{symm}=\bigoplus_{\y\in\P(p)} S_{\x+2\y}V.\]
 \item[(b)] $J_{\x,p}^{symm}$ is generated by its $\x$-isotypic component $S_{\x}V$.
\end{itemize}
We let $\mc{Q},\mc{R}$ denote the tautological quotient and sub- bundles on $\bb{G}=\bb{G}(p,V)$ and define
\[\S=\Sym(\Sym^2\mc{Q}),\ \S^{\vee}=\det(\Sym^2\mc{Q}^*)\oo\Sym(\Sym^2\mc{Q}^*),\rm{ and }\V=S_{\x^2}\mc{R}\oo S_{\x^1}\mc{Q},\]
where $\x^1=(x_1,\cdots,x_p)$ and $\x^2=(x_{p+1},\cdots,x_n)$. If we let $\mc{J}_{\x,p}^{symm}=\V\oo\S$ then
\[H^0(\bb{G},\mc{J}_{\x,p}^{symm})=J_{\x,p}^{symm},\ H^j(\bb{G},\mc{J}_{\x,p}^{symm})=0\rm{ for }j>0,\]
and moreover
 \[
  \Ext^{\bullet}_S(J_{\x,p}^{symm},S)=H^{{n+1\choose 2}-{p+1\choose 2}-\bullet}(\bb{G},\V\oo\mc{S}^{\vee})^*\oo\det(\Sym^2 W).
 \]
\end{lemma}

\begin{example}\label{ex:nonwitnessExt}
 Let $n=3$, $p=1$, and $\x=(2,2,0)$. We have $\bb{G}=\bb{P}^2$, $\mc{Q}=\mc{O}(1)$, and $\mc{R}=\Omega^1_{\bb{P}^2}(1)$. Moreover,
 \[\V=\Sym^2\mc{R}\oo\mc{Q}^2,\quad\S^{\vee}=\mc{Q}^{-2}\oo\Sym(\mc{Q}^{-2}),\]
 and therefore
 \[\V\oo\S^{\vee}=\bigoplus_{i\leq 0}\Sym^2\mc{R}\oo\mc{Q}^{2i}.\]
 Since $H^1(\bb{G},\Sym^2\mc{R})=\bw^2 V$, we get for $\bullet=4$ that $H^{{n+1\choose 2}-{p+1\choose 2}-\bullet}(\bb{G},\V\oo\mc{S}^{\vee})^*$ contains $\bw^2 W$ as a subrepresentation. Tensoring with $\det(\Sym^2 W)=S_{4,4,4}W$ and using the formula for $\Ext^{\bullet}_S(J_{\x,p}^{symm},S)$ in Lemma~\ref{lem:Jxpsymm} we obtain that 
 \[S_{5,5,4}W\rm{ is a subrepresentation of }\Ext^4_S(J_{\x,p}^{symm},S).\]
 One can show that in fact $J_{\x,p}=I_2/(I_2^2+I_3)$ where $I_2$ is the ideal generated by the $2\times 2$ minors of the generic symmetric $3\times 3$ matrix, while $I_3$ is the ideal generated by its determinant. A quick calculation with Macaulay2 \cite{M2} shows that $\Ext^4_S(J_{\x,p},S)$ vanishes in all degrees different from $-7$, and that the degree $-7$ component has dimension $3$. Since $\dim(S_{5,5,4}W)=\dim(\bw^2 W)=3$, it follows that in fact one has the equality $\Ext^4_S(J_{\x,p},S)=S_{5,5,4}W$.
\end{example}

\begin{proof}[Proof of Lemma~\ref{lem:Jxpsymm}]
 Consider the $\GL$-equivariant free $S$-module $M=S_{\x}V\oo S$. If follows from (\ref{eq:cauchy}) and the Littlewood-Richardson rule that
\[\scpr{S_{\x+2\y}V}{M}=1\rm{ for all }\y\in\P(p).\]
Moreover, if $\z\in\P(n)$ is such that $\scpr{S_{\z}V}{M}>0$ and $\z\neq\x+2\y$ for $\y\in\P(p)$ then the same rule implies that we must have $z_i>x_i$ for some $i=p+1,\cdots,n$. It follows that if we let $M_{\z}$ denote the $\z$-isotypic component of $M$, and define
\[N=\bigoplus_{\substack{\z\in\P(n) \\ \z\neq\x+2\y\rm{ for }\y\in\P(p)}} M_{\z},\]
then $N$ is a $\GL$-equivariant $S$-submodule of $M$, and moreover
\[M/N\simeq\bigoplus_{\y\in\P(p)} S_{\x+2\y}V\]
and $M/N$ is generated by $S_{\x}V$ since $M$ is. This proves the existence of a module with properties~(a) and~(b). For the uniqueness, observe that if $P$ satisfies (b) then there exists a surjective map $M\onto P$. If in addition it satisfies (a), then the kernel of this surjection is isomorphic to $N$ as a $\GL$-subrepresentation. Since there is no $N\neq N'\subset M$ such that $N\simeq N'$ as $\GL$-representations, it follows that $\rm{ker}(M\onto P)=N$ and $P\simeq M/N$. We write $J_{\x,p}^{symm}=M/N$.

To perform the $\Ext$ calculation, we consider $\mc{M}(\V)=\V\oo\S$, with the notation as in the statement of the lemma. We have by Bott's Theorem for Grassmannians \cite[Cor.~4.1.9]{weyman} that
\[H^0(\bb{G},\mc{M}(\V))=\bigoplus_{\y\in\P(p)} S_{\x+2\y}V,\rm{ and }H^j(\bb{G},\mc{M}(\V))=0\rm{ for }j>0.\]
If we can show that $H^0(\bb{G},\mc{M}(\V))$ is generated by $S_{\x}V$, then it must be isomorphic to $J_{\x,p}^{symm}$, and the calculation of $\Ext$ in terms of sheaf cohomology follows directly from \cite[Thm.~3.1]{raicu-weyman-witt}.

The tautological surjection $V\oo\mc{O}_{\bb{G}}\onto\mc{Q}$ induces a map $\Sym^2 V\oo\mc{O}_{\bb{G}}\onto\Sym^2\mc{Q}$, which in turn yields a map $\pi:S\oo\mc{O}_{\bb{G}}\onto\mc{S}$. $\pi$ induces a surjection $S\onto H^0(\bb{G},\S)$ on global sections (see for instance \cite[Section~6.3]{weyman}) and the structure of $H^0(\bb{G},\mc{M}(\V))$ as an $S$-module is defined by
\[S\oo H^0(\bb{G},\mc{M}(\V))\to H^0(\bb{G},\S)\oo H^0(\bb{G},\mc{M}(\V))\to H^0(\bb{G},\mc{M}(\V)).\]
To prove that $H^0(\bb{G},\mc{M}(\V))$ is generated by $S_{\x}V$, it is then enough to show that the multiplication
\[H^0(\bb{G},\S)\oo H^0(\bb{G},\V)\to H^0(\bb{G},\mc{M}(\V))\]
is surjective. Writing $\S=\bigoplus_{\y\in\P(p)}S_{2\y}\mc{Q}$, we have to check that the multiplication
\[H^0(\bb{G},S_{2\y}\mc{Q})\oo H^0(\bb{G},\V)\to H^0(\bb{G},S_{2\y}\mc{Q}\oo\V)\]
is surjective, but this is just given by the Cartan multiplication
\[S_{2\y}V\oo S_{\x}V\to S_{\x+2\y}V.\qedhere\]
\end{proof}

An analogous argument as in the proof of Lemma~\ref{lem:Jxpsymm} yields the following

\begin{lemma}\label{lem:Jxpskew}
 Let $S=\Sym\left(\bw^2 V\right)$ and $m=\lfloor n/2\rfloor$. For $1\leq p\leq m$ consider a partition $\x=(x_1=x_2=\cdots=x_{2p}\geq x_{2p+1}\geq\cdots\geq x_n\geq 0)$. There exists a unique $\GL$-equivariant $S$-module $J_{\x,p}^{skew}$ with the properties
\begin{itemize}
 \item[(a)] As a $\GL$-representation, $J_{\x,p}^{skew}$ has a decomposition
 \[J_{\x,p}^{skew}=\bigoplus_{\y\in\P(p)} S_{\x+\y^{(2)}}V.\]
 \item[(b)] $J_{\x,p}^{skew}$ is generated by its $\x$-isotypic component $S_{\x}V$.
\end{itemize}
We let $\mc{Q},\mc{R}$ denote the tautological quotient and sub- bundles on $\bb{G}=\bb{G}(2p,V)$ and define
\[\S=\Sym\left(\bw^2\mc{Q}\right),\ \S^{\vee}=\det\left(\bw^2\mc{Q}^*\right)\oo\Sym\left(\bw^2\mc{Q}^*\right),\rm{ and }\V=S_{\x^2}\mc{R}\oo S_{\x^1}\mc{Q},\]
where $\x^1=(x_1,\cdots,x_{2p})$ and $\x^2=(x_{2p+1},\cdots,x_n)$. If we let $\mc{J}_{\x,p}^{skew}=\V\oo\S$ then
\[H^0(\bb{G},\mc{J}_{\x,p}^{skew})=J_{\x,p}^{skew},\ H^j(\bb{G},\mc{J}_{\x,p}^{skew})=0\rm{ for }j>0,\]
and moreover
 \[
  \Ext^{\bullet}_S(J_{\x,p}^{skew},S)=H^{{n\choose 2}-{2p\choose 2}-\bullet}(\bb{G},\V\oo\mc{S}^{\vee})^*\oo\det\left(\bw^2 W\right).
 \]
\end{lemma}

In what follows we will write simply $J_{\x,p}$ instead of $J_{\x,p}^{skew}$ or $J_{\x,p}^{symm}$ when no confusion is possible.

\subsection{A spectral sequence for computing local cohomology}\label{subsec:spseq}

As before, we consider the vector space $X$ of symmetric (resp. skew-symmetric) $n\times n$ matrices, and let $S$ denote its coordinate ring $\Sym(\Sym^2 V)$ (resp. $\Sym(\bw^2 V)$). Let $Y_p\subset X$ denote the subvariety of matrices of rank at most $p$ (resp. $2p$). We write $m=\lfloor n/2\rfloor$ and define the $\GL$-equivariant $S$-modules
\begin{equation}\label{eq:defJpsymm}
 J_p = J_p^{symm} = \bigoplus_{\substack{\y\in\P(n) \\ y_1=\cdots=y_{p+1}}} J_{2\cdot\y,p},\rm{ for }p=0,\cdots,n-1,\rm{ when }S=\Sym(\Sym^2 V),
\end{equation}
respectively
\begin{equation}\label{eq:defJpskew}
 J_p = J_p^{skew} = \bigoplus_{\substack{\y\in\P(m) \\ y_1=\cdots=y_{p+1}}} J_{\y^{(2)},p},\rm{ for }p=0,\cdots,m-1,\rm{ when }S=\Sym\left(\bw^2 V\right).
\end{equation}

\begin{proposition}\label{prop:specseq}
 There exists a descending sequence $(I_r)_{r\geq 0}$ of ideals in $S$ and a spectral sequence
 \begin{equation}\label{eq:specseqgeneric}
  E^{i,j}_2=\Ext^{i-j}_S(I_j/I_{j+1},S)\Rightarrow\mc{H}^{i-j}_{Y_p}(X,\mc{O}_X),  
 \end{equation}
 where moreover
 \begin{equation}\label{eq:isomIr/Ir+1-Jp}
  \bigoplus_{r\geq 0} I_r/I_{r+1}\simeq J_p.  
 \end{equation}
\end{proposition}

\begin{remark}\label{rem:nondegeneratesseq}
 It can be shown that the spectral sequence in Proposition~\ref{prop:specseq} is degenerate when $S=\Sym\left(\bw^2 V\right)$, but this is not the case when $S=\Sym(\Sym^2 V)$ as the following example shows (this behavior is typical, and the spectral sequence is never degenerate as long as $0<p<n-1$). 
 
 Consider the situation of Example~\ref{ex:nonwitnessExt}: $n=3$, $p=1$. As shown there, $\Ext^4_S(J_p,S)$ contains a copy of the representation $S_{5,5,4}W$, which means that $S_{5,5,4}W$ appears on the $E_2$-page of the spectral sequence. However, a cancellation must occur when running the spectral sequence, since $S_{5,5,4}W$ doesn't occur as a subrepresentation in $\mc{H}^{\bullet}_{Y_p}(X,\mc{O}_X)$: this can be seen by noting that $S_{5,5,4}W$ is not a subrepresentation of any of the $\GL$-equivariant holonomic $\D$-modules on symmetric matrices (see Theorem~\ref{thm:charsDmods}(b) and~(\ref{eq:defBsCsj})).
\end{remark}

\begin{proof}[Proof of Proposition~\ref{prop:specseq}]
 We write $\P=\P(n)$ if $S=\Sym(\Sym^2 V)$, and $\P=\P(m)$ when $S=\Sym\left(\bw^2 V\right)$. We choose a total ordering of the partitions $\ll\in\P$ with $\ll_1=\cdots=\ll_{p+1}$:
\begin{equation}\label{eq:orderP}
 \ll(0),\ll(1),\cdots,\ll(r),\cdots
\end{equation}
such that there exists no $i<j$ with $\ll(i)\geq\ll(j)$. Using (\ref{eq:defIz}), we define a decreasing sequence of ideals
\[S=I_0\supset I_1\supset\cdots\supset I_r\supset\cdots,\rm{ by}\]
\[I_r=\sum_{i\geq r}I_{\ll(i)}.\]
We fix $r$ and define
\[\x=2\ll(r)\rm{ if }S=\Sym(\Sym^2 V),\rm{ and }\x=\ll(r)^{(2)}\rm{ if }S=\Sym\left(\bw^2 V\right).\]
We will prove that $I_r/I_{r+1}\simeq J_{\x,p}$, which then implies (\ref{eq:isomIr/Ir+1-Jp}). The existence of the spectral sequence (\ref{eq:specseqgeneric}) follows as in the proof of \cite[Thm.~4.1]{raicu-VeroDmods}.

To prove $I_r/I_{r+1}\simeq J_{\x,p}$ we need to check that $I_r/I_{r+1}$ satisfies the assumptions (a) and (b) of Lemma~\ref{lem:Jxpsymm} resp.~\ref{lem:Jxpskew}. Part (a) is a consequence of the equality (\ref{eq:descIzs}). Part (b) follows from the fact that $I_r/I_{r+1}$ is a quotient of $I_{\ll(r)}$, and $I_{\ll(r)}$ is generated by $S_{\x}V$.
\end{proof}

\section{Skew-symmetric matrices}\label{sec:skew}

In this section $W=\bb{C}^n$, $V=W^*$, and $S=\Sym(\bw^2 V)$ is the coordinate ring of the vector space $\bw^2 W$ of $n\times n$ skew-symmetric matrices.

\begin{theorem}\label{thm:skew}
 Let $m=\lfloor n/2\rfloor$ and fix $0\leq p<m$. With notation (\ref{eq:defqbinomial}), (\ref{eq:genfcngeneric}), (\ref{eq:defllBs}) and (\ref{eq:defJpskew}), we have
 \[\scpr{\Ext^{\bullet}_S(J_p^{skew},S)}{S_{\ll(\mf{B}_s)}W}=
 \begin{cases}
  \displaystyle q^{2(m-p)^2-(m-p)+2s}\cdot{s-1\choose s-(m-p)}_{q^4} & \rm{if }m-p\leq s\leq m,\rm{ and }n=2m+1, \\
  \displaystyle q^{2(m-p)^2-(m-p)}\cdot{s-1\choose s-(m-p)}_{q^4} & \rm{if }m-p\leq s\leq m,\rm{ and }n=2m, \\
  0 & \rm{otherwise}.
 \end{cases}
 \]
\end{theorem}

We begin by showing how to use this theorem in order to prove (\ref{eq:Hpqskew}).

\begin{proof}[Proof of~(\ref{eq:Hpqskew})]
 We fix $0\leq p<m$, and follow the strategy for proving the Main Theorem outlined in the Introduction. We consider the equation (\ref{eq:DmodHpqtoFi}), where $I=\{0,\cdots,m\}$ and for $s\in I$, $F_s$ is the $\D$-module with character $\mf{B}_s$. Since $\ll(\mf{B}_s)$ is a witness weight for $F_s$, we can detect the polynomials $P_s(q)$ by
 \[P_s(q)=\scpr{\mc{H}^{\bullet}_{Y_p}(X,\mc{O}_X)}{S_{\ll(\mf{B}_s)}W}.\]
 We now use the spectral sequence (\ref{eq:specseqgeneric}) to analyze $\mc{H}^{\bullet}_{Y_p}(X,\mc{O}_X)$. It follows from (\ref{eq:isomIr/Ir+1-Jp}) with $J_p=J_p^{skew}$, and from Theorem~\ref{thm:skew} that $S_{\ll(\mf{B}_s)}W$ can appear as a subrepresentation of $E_2^{i,j}$ only when $i-j\equiv m-p\ (\rm{mod }2)$. It follows that as we run the spectral sequence, no cancellations can occur between copies of $S_{\ll(\mf{B}_s)}W$, and therefore
 \[P_s(q)=\scpr{\mc{H}^{\bullet}_{Y_p}(X,\mc{O}_X)}{S_{\ll(\mf{B}_s)}W}=\scpr{\Ext^{\bullet}_S(J_p^{skew},S)}{S_{\ll(\mf{B}_s)}W}.\]
 To finish the proof of (\ref{eq:Hpqskew}), we just need to note that the intersection homology $\D$-module $D_s$ is $F_{m-s}$ by Theorem~\ref{thm:charsofICmods}, so the coefficient of $[D_s]$ in $H_p(q)$ is
 \[P_{m-s}(q)=
  \begin{cases}
  \displaystyle q^{2(m-p)^2-(m-p)+2(m-s)}\cdot{m-s-1\choose m-s-(m-p)}_{q^4} & \rm{if }0\leq s\leq p,\rm{ and }n=2m+1, \\
  \displaystyle q^{2(m-p)^2-(m-p)}\cdot{m-s-1\choose m-s-(m-p)}_{q^4} & \rm{if }0\leq s\leq p,\rm{ and }n=2m, \\
  0 & \rm{otherwise}.
 \end{cases}
 \]
This formula is precisely what (\ref{eq:Hpqskew}) predicts, which concludes our proof.
\end{proof}

\begin{proof}[Proof of Theorem~\ref{thm:skew}]
 
We begin by fixing some notation. For $d\geq 0$ we consider a partition $\y\in\P(m)$ with $y_1=\cdots=y_{p+1}=d$. We let $\x=\y^{(2)}\in\mc{P}(n)$ (setting $x_n=0$ if $n$ is odd) and write $\x$ as the concatenation of $\x^1=(d^{2p})\in\P(2p)$ and $\x^2\in\P(n-2p)$. We consider the partition $\b=\x^2+((n-1-2s)^{n-2p})\in\P(n-2p)$, so that
\begin{equation}\label{eq:bfromy-skew}
\b_1=\b_2=d+n-1-2s,\ \b_{2i-1}=\b_{2i}=y_{p+i}+n-1-2s\rm{ for }i=2,\cdots,m-p.
\end{equation}
For $s=m-p,\cdots,m$ we consider the collection of dominant weights
\begin{equation}\label{eq:defAs-skew}
\mc{A}_s=\{\a\in\bb{Z}^{2p}_{dom}:\a_{2i-1}=\a_{2i}\rm{ for }i=1,\cdots,p,\rm{ and }\a_1\leq d+n-2s-2p\}, 
\end{equation}
We define the cohomologically graded module
\begin{equation}\label{eq:defH*-skew}
H^{\bullet}=H^{{n\choose 2}-{2p\choose 2}-\bullet}(\bb{G},S_{\b}\mc{R}\oo\det(\mc{Q})^{\oo(d+n-2s-2p)}\oo\Sym(\bw^2\mc{Q}^*)) 
\end{equation}
with $\bb{G},\mc{Q},\mc{R}$ as in Lemma~\ref{lem:Jxpskew}. Our proof is based on a number of claims which we explain at the end:

\ul{Claim 1:} For $J_{\x,p}^{skew}$ as in Lemma~\ref{lem:Jxpskew} and $s=m-p,\cdots,m$,
\[
\scpr{\Ext^{\bullet}_S(J_{\x,p}^{skew},S)}{S_{\ll(\mf{B}_s)}W} = \scpr{H^{\bullet}}{\bb{C}}.
\]

\ul{Claim 2:}
\[\scpr{H^{\bullet}}{\bb{C}}=\begin{cases}
 q^{{n\choose 2}-{2p\choose 2}-|\b|} & \rm{if }\b\in\P(n-2p,2p)\rm{ and }(\b')^*\in\mc{A}_s, \\
 0 & \rm{otherwise}.
\end{cases}\]

\ul{Claim 3:} For $\b\in\P(n-2p,2p)$ satisfying (\ref{eq:bfromy-skew}), we have the equivalences

\[(\b')^*\in\mc{A}_s\Longleftrightarrow\begin{cases}
d=2s+2p-n+1, \\
\b_i\rm{ is even for all }i=1,\cdots,n-2p.
\end{cases}
\]

We now explain the proof of the theorem based on Claims~1--3. It follows from Claims~2 and~3 that the only partitions $\y\in\P(m)$ for which the corresponding $\b$ (as in (\ref{eq:bfromy-skew})) yields $\scpr{H^{\bullet}}{\bb{C}}\neq 0$, satisfy
\begin{equation}\label{eq:skewcondy}
y_1=\cdots=y_{p+1}=2s+2p-n+1,\quad y_i\equiv n-1\ (\rm{mod }2). 
\end{equation}
Any such $\y$ is obtained from a unique partition $\z\in\P(m-p-1,s+p-m)$ via
\begin{equation}\label{eq:ztoy}
y_{p+1+i}=\begin{cases}
 2z_i & \rm{if }n\rm{ is odd,} \\
 2z_i+1 & \rm{if }n\rm{ is even,}
\end{cases}
\quad\rm{ for }i=1,\cdots,m-p-1.
\end{equation}
Since $|\b|=(n-1-2s)(n-2p)+2(y_{p+1}+\cdots+y_m)$, we get that $|\b|$ and $|\z|$ are related via
\[
|\b|=\begin{cases}
 (2m-2s)(2m+1-2p)+2(2s+2p-2m)+4|\z| & \rm{ if }n=2m+1\rm{ is odd}; \\
 (2m-1-2s)(2m-2p)+2(2s+2p-2m+1)+2(m-p-1)+4|\z| & \rm{ if }n=2m\rm{ is even}. \\
\end{cases}
\]
Equivalently, after some easy manipulations we have
\[
{n\choose 2}-{2p\choose 2}-|\b| =\begin{cases}
 2(m-p)^2-(m-p)+2s+4((m-p-1)(s-(m-p))-|\z|) & \rm{ if }n=2m+1\rm{ is odd}; \\
 2(m-p)^2-(m-p)+4((m-p-1)(s-(m-p))-|\z|) & \rm{ if }n=2m\rm{ is even}. \\
\end{cases}
\]

Combining Claims~1--3 with the above equality and the fact that
\[\sum_{\z\in\P(m-p-1,s-(m-p))}q^{4((m-p-1)(s-(m-p))-|\z|)}\overset{(\ref{eq:genfunqbincomplement})}{=} {s-1\choose s-(m-p)}_{q^4}, \]
we obtain the desired conclusion. To finish the proof of the theorem, it remains to explain Claims~1--3.

\ul{Proof of Claim 1:} The following easy identities will be useful next:
\begin{equation}\label{eq:basicrelns-skew1}
\det\left(\bw^2 W\right)=\det(W)^{\oo(n-1)},\quad S_{\ll(\mf{B}_s)}W=\det(W)^{\oo(2s)},
\end{equation}
\begin{equation}\label{eq:basicrelns-skew2}
\det(V)\oo\mc{O}_{\bb{G}}=\det(\mc{Q})\oo\det(\mc{R}),\quad\det\left(\bw^2\mc{Q}^*\right)=\det(\mc{Q})^{\oo(-2p+1)},\quad S_{\x^1}\mc{Q}=\det(\mc{Q})^{\oo d}.
\end{equation}
By Lemma~\ref{lem:Jxpskew} (and the notation thereof) we get
\[
\begin{aligned}
&\scpr{\Ext^{\bullet}_S(J_{\x,p}^{skew},S)}{S_{\ll(\mf{B}_s)}W} \overset{(\ref{eq:basicrelns-skew1})}{=} \scpr{H^{{n\choose 2}-{2p\choose 2}-\bullet}(\bb{G},\V\oo\S^{\vee})^*\oo\det(W)^{\oo(n-1)}}{\det(W)^{\oo(2s)}} \\
&\overset{\rm{tensor with }\det(W)^{\oo(-2s)}}{=} \scpr{H^{{n\choose 2}-{2p\choose 2}-\bullet}(\bb{G},\V\oo\S^{\vee})^*\oo\det(W)^{\oo(n-1-2s)}}{\bb{C}} \\
&\overset{\rm{dualize}}{=}\scpr{H^{{n\choose 2}-{2p\choose 2}-\bullet}(\bb{G},\V\oo\S^{\vee})\oo\det(V)^{\oo(n-1-2s)}}{\bb{C}} \overset{(\ref{eq:bfromy-skew}),(\ref{eq:defH*-skew}),(\ref{eq:basicrelns-skew2})}{=}\scpr{H^{\bullet}}{\bb{C}}.
\end{aligned}
\]

\ul{Proof of Claim 2:} We have
\[\det(\mc{Q})^{\oo(d+n-2s-2p)}\oo\Sym(\bw^2\mc{Q}^*)=\bigoplus_{\a\in\mc{A}_s} S_{\a}\mc{Q},\]
and therefore
\[H^{\bullet}=\bigoplus_{\a\in\mc{A}_s}H^{{n\choose 2}-{2p\choose 2}-\bullet}(\bb{G},S_{\b}\mc{Q}\oo S_{\a}\mc{R}).\]
The asserted conclusion now follows from Theorem~\ref{thm:bott}(a).

\ul{Proof of Claim 3:} Assume first that $(\b')^*=\a\in\mc{A}_s$. Since $\a_{2i-1}=\a_{2i}$ for $i=1,\cdots,p$, we get that all $\b_i$ are even. Since $\b\in\P(n-2p,2p)$, we have $d+n-1-2s=\b_1\leq 2p$. If the inequality is strict, then $\b'_{2p}=0$ and therefore $\a_1=-\b'_{2p}=0$. Since $\a\in\mc{A}_s$, we get $0=\a_1\leq d+n-2s-2p$, which then implies $d+n-2s\geq 2p$. Since $\b_1$ is even and $\b_1<2p$, we get $2p-2\geq\b_1=d+n-1-2s$, i.e. $2p\geq d+n+1-2s$, contradicting the inequality $d+n-2s\geq 2p$.

Assume now that $d=2s+2p-n+1$, and that all $\b_i$ are even. $\mc{A}_s$ simplifies to
\[\mc{A}_s=\{\a\in\bb{Z}^{2p}_{dom}:\a_{2i-1}=\a_{2i}\rm{ for }i=1,\cdots,p,\rm{ and }\a_1\leq 1\}.\]
Let $\a=(\b')^*$. Since the $\b_i$ are even, we get $\a_{2i-1}=\a_{2i}$ for $i=1,\cdots,p$. Since $\a_1=-\b'_{2p}\leq 0$, the inequality $\a_1\leq 1$ is trivially satisfied, proving that $(\b')^*\in\mc{A}_s$.
\end{proof}

\section{Symmetric matrices}\label{sec:symm}

In this section $W=\bb{C}^n$, $V=W^*$, and $S=\Sym(\Sym^2 V)$ is the coordinate ring of the vector space $\Sym^2 W$ of $n\times n$ symmetric matrices.

\begin{theorem}\label{thm:symm}
We fix $0\leq p<n$, $n-p\leq s\leq n$. Using the notation (\ref{eq:defqbinomial}), (\ref{eq:genfcngeneric}), (\ref{eq:defllCsj}) and (\ref{eq:defJpsymm}), we have
 \[\scpr{\Ext^{\bullet}_S(J_p^{symm},S)}{S_{\ll(\mf{C}_s^j)}W}=
 \begin{cases}
  \displaystyle q^{1+{s+1\choose 2}-{s-(n-p)+2\choose 2}}\cdot{\lfloor\frac{s-1}{2}\rfloor\choose\frac{s-(n-p)}{2}}_{q^{-4}} & \rm{if }s\equiv n-p\ (\rm{mod }2),\rm{ and} \\
  & j\equiv s\ (\rm{mod }2)\rm{ when }s<n;\\
  0 & \rm{otherwise}.
 \end{cases}
 \]
\end{theorem}

The identity (\ref{eq:Hpqsymm}) follows from Theorem~\ref{thm:symm} by the same reasoning for which (\ref{eq:Hpqskew}) was a consequence of Theorem~\ref{thm:skew}. We leave the details to the interested reader, and focus on the proof of Theorem~\ref{thm:symm}.

\begin{proof}[Proof of Theorem~\ref{thm:symm}]
We begin by fixing some notation. For $d\geq 0$ we consider a partition $\y\in\P(n)$ with $y_1=\cdots=y_{p+1}=d$. We let $\x=2\cdot\y\in\mc{P}(n)$ and write $\x$ as the concatenation of $\x^1=((2d)^p)\in\P(p)$ and $\x^2\in\P(n-p)$. We consider the partition $\b=\x^2+((n-s)^{n-p})\in\P(n-p)$, so that
\begin{equation}\label{eq:bfromy-symm}
\b_1=2d+n-s,\ \b_i=2y_{p+i}+n-s\rm{ for }i=2,\cdots,n-p. 
\end{equation}
For $s=n-p,\cdots,n$ we consider the collection of dominant weights
\begin{equation}\label{eq:defAs-symm}
\mc{A}_s=\{\a\in\bb{Z}^p_{dom}:\a_i\equiv n-s-p-1\ (\rm{mod }2),\a_1\leq 2d+n-s-p-1\}, 
\end{equation}
We define the cohomologically graded module
\begin{equation}\label{eq:defH*-symm}
H^{\bullet}=H^{{n+1\choose 2}-{p+1\choose 2}-\bullet}(\bb{G},S_{\b}\mc{R}\oo\det(\mc{Q})^{\oo(2d+n-s-p-1)}\oo\Sym(\Sym^2\mc{Q}^*)) 
\end{equation}
with $\bb{G},\mc{Q},\mc{R}$ as in Lemma~\ref{lem:Jxpsymm}. Our proof is based on a number of claims which we explain at the end:

\ul{Claim 1:} For $J_{\x,p}^{symm}$ as in Lemma~\ref{lem:Jxpsymm} and $s=n-p,\cdots,n$,
\[
\scpr{\Ext^{\bullet}_S(J_{\x,p}^{symm},S)}{S_{\ll(\mf{C}_s^j)}W} = \begin{cases}
\scpr{H^{\bullet}}{\bb{C}} & \rm{if }j=1, \\
\scpr{H^{\bullet}}{\bw^{n-s}W} & \rm{if }j=2.
\end{cases}
\]

\ul{Claim 2:} For $\b$ satisfying (\ref{eq:bfromy-symm}), we have
\[\scpr{H^{\bullet}}{\bb{C}}=\begin{cases}
 q^{{n+1\choose 2}-{p+1\choose 2}-|\b|} & \rm{if }\b\in\P(n-p,p)\rm{ and }(\b')^*\in\mc{A}_s, \\
 0 & \rm{otherwise}.
\end{cases}\]

\[\scpr{H^{\bullet}}{\bw^{n-s}W}=\begin{cases}
 q^{{n+1\choose 2}-{p+1\choose 2}-|\b|} & \rm{if }\b\in\P(n-p,p)\rm{ and }(\b')^*+(0^{s-n+p},-1^{n-s})\in\mc{A}_s, \\
 0 & \rm{otherwise}.
\end{cases}\]

\ul{Claim 3:} For $\b\in\P(n-p,p)$ satisfying (\ref{eq:bfromy-symm}), we have the equivalences

\[(\b')^*\in\mc{A}_s\Longleftrightarrow\begin{cases}
2d=s+p-n, \\
\b'_i\rm{ odd for }i=n-s+1,\cdots,p,\\
n-p\rm{ is odd when }s<n.
\end{cases}
\]

\[(\b')^*+(0^{s-n+p},-1^{n-s})\in\mc{A}_s\Longleftrightarrow\begin{cases}
2d=s+p-n, \\
\b'_i\rm{ odd for }i=n-s+1,\cdots,p,\\
n-p\rm{ is even when }s<n.
\end{cases}
\]

We now explain the proof of the theorem based on Claims~1--3. The condition in Claim~3 that $\b'_i$ is odd for $i=n-s+1,\cdots,p$ is equivalent with the following condition on the original partition $\y$:
\[y_{p+2}=y_{p+3},y_{p+4}=y_{p+5},\cdots,\rm{ and moreover }y_n=0\rm{ if }n-p-1\rm{ is odd.}\]
This is equivalent to the existence of a partition $\z\in\P(\lfloor(n-p-1)/2\rfloor)$ such that $(y_{p+2},\cdots,y_n)=\z^{(2)}$. In fact, the non-trivial contributions to $H^{\bullet}$ in Claim~2 occur when $(s+p-n)/2=d=y_{p+1}\geq y_{p+2}\geq\cdots$, i.e. for $\z\in\P(\lfloor(n-p-1)/2\rfloor,(s+p-n)/2)$. Note that $|\b|$ and $|\z|$ are then related via
\begin{equation}\label{eq:bfromz}
|\b|=(n-s)\cdot(n-p)+2d+2(y_{p+2}+\cdots+y_n)=(n-s)\cdot(n-p)+(s+p-n)+4|\z|. 
\end{equation}
Combining Claims~1--3 with (\ref{eq:bfromz}) we get that the following equality holds when $s+p-n\geq 0$ is even, and any of the following conditions (a)--(c) holds: (a)~$j=1$ and $n-p$ is odd; (b)~$j=2$ and $n-p$ is even; (c)~$s=n$, independently on the parity of $j,n-p$ (note that this is compatible with the equality $\mf{C}_n^1=\mf{C}_n^2$).
\[
\begin{aligned}
\scpr{\Ext^{\bullet}_S(J_{p}^{symm},S)}{S_{\ll(\mf{C}_s^j)}W} &= \sum_{\z\in\P(\lfloor(n-p-1)/2\rfloor,(s+p-n)/2)}q^{{n+1\choose 2}-{p+1\choose 2}-(n-s)\cdot(n-p)-(s+p-n)-4|\z|} \\
&\overset{(\ref{eq:genfunqbin})}{=} q^{{n+1\choose 2}-{p+1\choose 2}-(n-s)\cdot(n-p)-(s+p-n)}\cdot{\lfloor\frac{s-1}{2}\rfloor \choose \frac{s-(n-p)}{2}}_{q^{-4}} \\
\end{aligned}
\]
The formula asserted in the statement of the theorem now follows by observing that 
\[{n+1\choose 2}-{p+1\choose 2}-(n-s)\cdot(n-p)-(s+p-n)=1+{s+1\choose 2}-{s-(n-p)+2\choose 2}.\]
To finish the proof of the theorem, it remains to explain Claims~1--3.

\ul{Proof of Claim 1:} The following easy identities will be useful next:
\begin{equation}\label{eq:basicrelns-symm}
\det\left(\Sym^2 W\right)=\det(W)^{\oo(n+1)},\quad\det(\Sym^2\mc{Q}^*)=\det(\mc{Q})^{\oo(-p-1)},\quad\det(V)\oo\mc{O}_{\bb{G}}=\det(\mc{Q})\oo\det(\mc{R}). 
\end{equation}
By Lemma~\ref{lem:Jxpsymm} (and the notation thereof) we get
\[
\begin{aligned}
&\scpr{\Ext^{\bullet}_S(J_{\x,p}^{symm},S)}{S_{\ll(\mf{C}_s^j)}W} \overset{(\ref{eq:basicrelns-symm})}{=} \scpr{H^{{n+1\choose 2}-{p+1\choose 2}-\bullet}(\bb{G},\V\oo\S^{\vee})^*\oo\det(W)^{\oo(n+1)}}{S_{\ll(\mf{C}_s^j)}W} \\
&\overset{\rm{tensor with }\det(W)^{\oo(-s-1)}}{=} \scpr{H^{{n+1\choose 2}-{p+1\choose 2}-\bullet}(\bb{G},\V\oo\S^{\vee})^*\oo\det(W)^{\oo(n-s)}}{S_{\ll(\mf{C}_s^j)-((s+1)^n)}W} \\
&\overset{\rm{dualize}}{=}\scpr{H^{{n+1\choose 2}-{p+1\choose 2}-\bullet}(\bb{G},\V\oo\S^{\vee})\oo\det(V)^{\oo(n-s)}}{S_{\ll(\mf{C}_s^j)-((s+1)^n)}V} \overset{(\ref{eq:bfromy-symm}),(\ref{eq:defH*-symm}),(\ref{eq:basicrelns-symm})}{=}\scpr{H^{\bullet}}{S_{\ll(\mf{C}_s^j)-((s+1)^n)}V}
\end{aligned}
\]
For $j=1$ we have $\ll(\mf{C}_s^1)=(s+1)^n$, so $S_{\ll(\mf{C}_s^1)-((s+1)^n)}V=\bb{C}$. For $j=2$ we have $\ll(\mf{C}_s^2)-((s+1)^n)=(0^s,-1^{n-s})$ so $S_{\ll(\mf{C}_s^j)-((s+1)^n)}V=\bw^{n-s}W$.

\ul{Proof of Claim 2:} We have
\[\det(\mc{Q})^{\oo(2d+n-s-p-1)}\oo\Sym(\Sym^2\mc{Q}^*)=\bigoplus_{\a\in\mc{A}_s} S_{\a}\mc{Q},\]
and therefore
\[H^{\bullet}=\bigoplus_{\a\in\mc{A}_s}H^{{n+1\choose 2}-{p+1\choose 2}-\bullet}(\bb{G},S_{\b}\mc{Q}\oo S_{\a}\mc{R}).\]
The asserted conclusions now follow from Theorem~\ref{thm:bott}.

\ul{Proof of Claim 3:} Let us first prove that either of the assumptions $(\b')^*\in\mc{A}_s$ or $(\b')^*+(0^{s-n+p},-1^{n-s})$ implies that $2d=s+p-n$, or equivalently $\b_1=p$. Suppose otherwise that $\b_1<p$, or equivalently that $\b'_p=0$, and that $(\b')^*=\a$ or $(\b')^*+(0^{s-n+p},-1^{n-s})=\a$ for some $\a\in\mc{A}_s$. We get in either case that $0=-\b'_p=\a_1\leq 2d+n-s-p-1$, i.e. $p+1\leq 2d+n-s=\b_1$, a contradiction. To prove the equivalences in Claim~3 we can then assume that $2d=s+p-n$ (and $\b_1=p$, $\b'_p\geq 1$). Formula (\ref{eq:defAs-symm}) then simplifies to
\[\mc{A}_s=\{\a\in\bb{Z}^p_{dom}:\a_i\rm{ odd},\a_1\leq -1\}.\]
Let us now observe that
\[(\b')^*=(-\b'_p,\cdots,-\b'_{n-s+1},(p-n)^{n-s}),\]
so the first entry in $(\b')^*$ resp. $(\b')^*+(0^{s-n+p},-1^{n-s})$ is at most $-1$. Containment in $\mc{A}_s$ for either of the two weights $(\b')^*$ or $(\b')^*+(0^{s-n+p},-1^{n-s})$ is then equivalent to them having only odd entries. We get $(\b')^*\in\mc{A}_s$ if and only if $\b'_i$ is odd for $i=n-s+1,\cdots,p$, and $n-p$ is odd when $s<n$. Also, $(\b')^*+(0^{s-n+p},-1^{n-s})\in\mc{A}_s$ if and only if $\b'_i$ is odd for $i=n-s+1,\cdots,p$, and $n-p$ is even when~$s<n$.
\end{proof}

\section*{Acknowledgments} 
We would like to thank Gennady Lyubeznik, Anurag Singh, and Uli Walther for interesting conversations and helpful suggestions. Experiments with the computer algebra software Macaulay2 \cite{M2} have provided numerous valuable insights. The first author acknowledges the support of the National Science Foundation Grant No.~1458715. The second author acknowledges the support of the Alexander von Humboldt Foundation, and of the National Science Foundation Grant No.~1400740.


\end{document}